\documentclass[14pt]{article}
\usepackage{amsmath}
\usepackage{amssymb}
\usepackage{amsthm}
\usepackage{csquotes}
\usepackage{epsfig}
\usepackage{graphicx}

\newtheorem{theorem}{Theorem}[section]
\newtheorem{lemma}[theorem]{Lemma}
\newtheorem{corollary}[theorem]{Corollary}
\newtheorem{proposition}[theorem]{Proposition}

\newtheorem{remark}[theorem]{Remark}

\newtheorem{definition}[theorem]{Definition}

\numberwithin{equation}{section}

\numberwithin{equation}{section}

\renewcommand{\phi}{\varphi}

\usepackage{xcolor}

 \newcommand{\mc}[1]{{\mathcal #1}}

 \newcommand{\bb}[1]{{\mathbb #1}}

\newcommand{\p}{\partial}

\newcommand{\pfrac}[2]{\genfrac{}{}{}{1}{#1}{#2}}
\newcommand{\ppfrac}[2]{\genfrac{}{}{}{2}{#1}{#2}}
\newcommand{\at}[2]{\genfrac{}{}{0pt}{}{#1}{#2}}

\begin{document}

\title{Diffusion Processes: entropy, Gibbs states,  and  the continuous time Ruelle  operator}
\author{A. O.  Lopes,  G. Muller, and A. Neumann \\  Instituto de Matem\'atica e Estat\'istica,\\ UFRGS, Porto Alegre, Brasil}
\smallskip

\date{\today}

\maketitle

\begin{abstract}

We consider a   Riemannian compact manifold $M$, the associated Laplacian $\Delta$ and the corresponding Brownian motion $X_t$, $t\geq 0.$
Given a Lipschitz function $V:M\to\bb R$ we consider the operator $\frac{1}{2}\Delta+V$,
which acts on differentiable functions $f: M\to\bb R$ via the expression
$$\frac{1}{2} \Delta f(x)+\,V(x)f(x) \,,$$
for all $x\in M$.

Denote by $P_t^V$, $t \geq 0,$ the semigroup acting on functions $f: M\to\bb R$ given by
$$P_{t}^V (f)(x)\,:=\, \bb E_{x} \big[e^{\int_0^{t} V(X_r)\,dr} f(X_t)\big].\,$$

We will derive results that show that this semigroup is a continuous-time version of the discrete-time Ruelle operator.

Consider the positive differentiable eigenfunction $F: M \to \mathbb{R}$ associated with the main eigenvalue $\lambda$,
for the semigroup $P_t^V$, $t \geq 0$. From the function $F$, in a procedure similar to the one used in discrete-time Thermodynamic Formalism, we can associate by way of a coboundary procedure, a certain stationary Markov semigroup.
We show that the probability on the Skorohod space obtained from this new stationary Markov semigroup meets the requirements to be called stationary Gibbs state associated with the potential $V$. We define entropy, pressure, and the continuous-time Ruelle operator. Also, we present a variational principle of pressure for such a setting.
\vspace{0.2cm}\newline
\vspace{0.2cm}\newline

\medskip

Emails: $arturoscar.lopes@gmail.com$

$gustavo.muller\_ nh@hotmail.com$ 

$neumann.adri@gmail.com$

\end{abstract}

In the present text, we will work
with $\{X_t;\,t\geq 0\}$ the Brownian Motion with state-space on a Riemannian compact manifold $M$. One particular example could be $\bb S^1$, where $\bb S^1$ is the interval $[0,1]$ with $0\equiv 1$, or a $n$ dimensional torus. We call $M$ the state space. In order to simplify the notation we will assume that $M=\mathbb{S}^1$,  equipped with the Euclidean metric (for the general case similar results can be obtained, but then we would get more cumbersome expressions).

We will denote by $\frac{1}{2}\pfrac{\partial^2}{\partial x^2} $ the Laplacian on the Riemannian manifold $\mathbb{S}^1$. The Laplacian operator is selfadjoint (see \cite{Stri}). As mentioned in \cite{Bov} the associated stochastic process is reversible and this will play an important role here.



The Brownian Motion has as its infinitesimal generator, the Laplacian, which is the operator $L=\frac{1}{2}\Delta=\frac{1}{2}\pfrac{\partial^2}{\partial x^2}$ acting on functions $f\in C^2( M)$.
The trajectories of this process are in $\mc C$, the space of all continuous functions defined in $[0,T]$ taking values in $M$.

{\bf{Denote $\bb P_\mu$ the probability in $\mc C$ induced by $\{X_t;\,t\geq 0\}$ and the initial probability $\mu$. If the initial measure is $\delta_x$ (Dirac's measure), for some $x\in M$, we will denote $\bb P_{\delta_x}$ only by $\bb P_{x}$.
Note that we can related $\bb P_\mu$, for any initial probability $\mu$ on $M$, as $\bb P_{x}$ through the next integral over $M$:
\begin{equation}\label{explain_measure}
    \bb P_\mu[A]=\int_{M} \bb P_{x}[A]\, d\mu(x),
\end{equation}
for all measurable set A in the Skorohod space.
Moreover, the expectation (integral) with respect to $\bb P_{\mu}$ or $\bb P_{x}$ will be denoted by
$\bb E_{\mu}$ or $\bb E_{x}$, respectively.}}

 {\bf  The Skorohod space $\mathcal{C}$ is the set of paths $w:[0,\infty) \to M $ which  are c\`adl\`ag (see  \cite{Bov} or \cite{EK}).  We are interested in probabilities on this set.
For a fixed $T>0$,
sometimes it  is natural to consider the restriction of $\mathcal{C}$ to the subset  of paths of the form $\{\,w:[0,T) \to M \,\}$. As we are considering diffusions, with probability $1$ the paths in $\mathcal{C}$ are continuous (see \cite{AGM}), and then,   from now on we will denote the set of paths by $\mathcal{C}$,  in accord with the above notation.}

We define for each fixed $s\in \mathbb{R}^{+}\cup \{0\}$ the
$\mathcal{ B}$-measurable transformation $\Theta_s: \mathcal{C} \to \mathcal{C}$
given by $\Theta_s (w_t) = w_{t+s}$. $\Theta_s$, $s \geq 0,$ is called the continuous time shift. A stationary diffusion process defines a probability on $\mathcal{C}$, which is invariant for the continuous time flow
$\Theta_s: \mathcal{C} \to \mathcal{C}$, $s \geq 0$ (see \cite{Loche} for instance)

\begin{figure}[!htb]
	\centering
	{\includegraphics[scale=0.77]{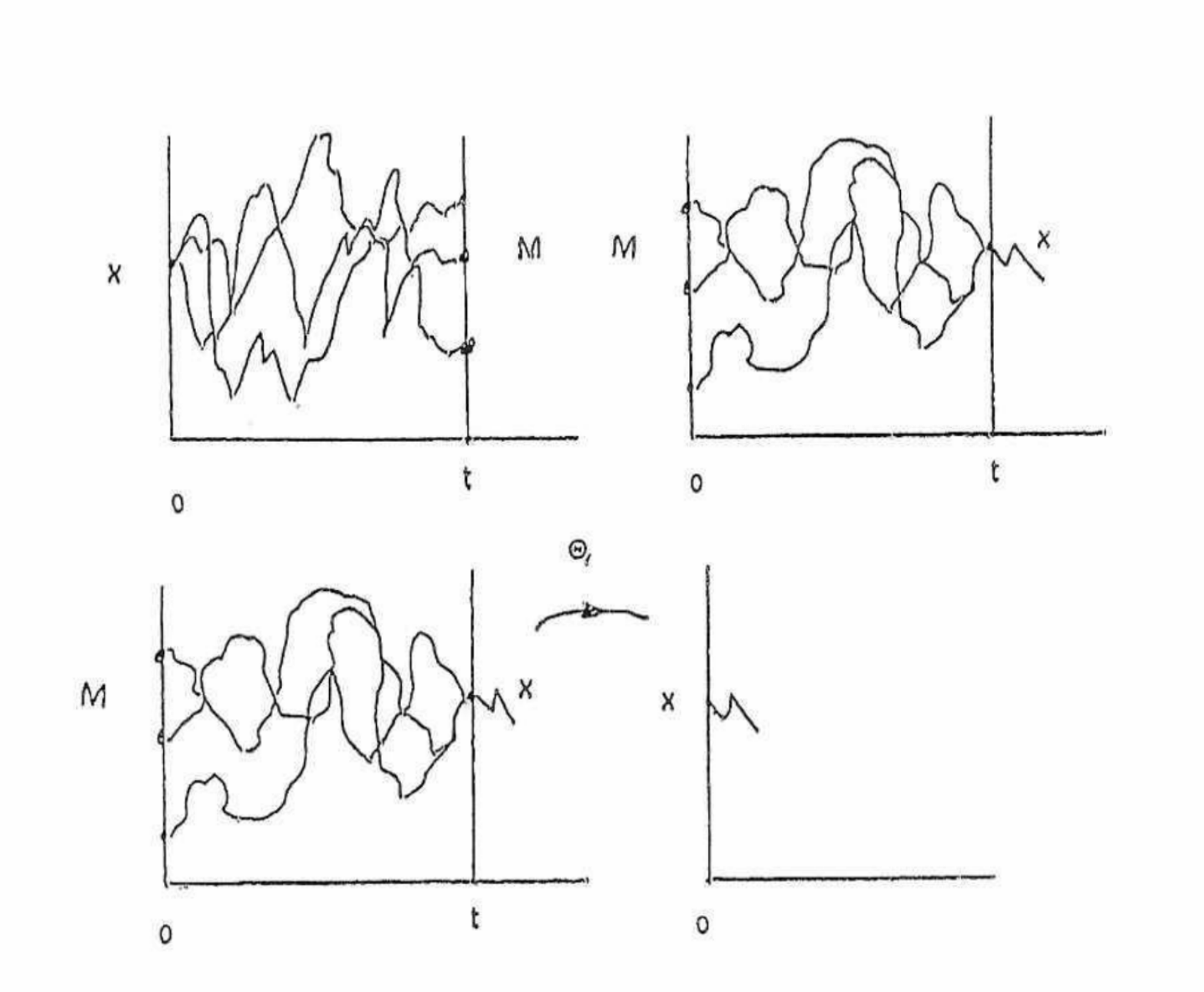}}
\caption{In the left top figure, the point $x\in M$ is the value at time $t=0$ of the path obtained as the image - by the continuous time shift $\Theta_t$ - of the set of paths described above on the left figure. In the top right figure the point $x\in M$ is the value at time $t$, and we exhibit a set of paths that are preimages of a path (which takes the value $x$ at time $t$) by the shift $\Theta_t$. The bottom figure describes our claim in a more precise form.}
\label{finalFK}
\end{figure}

{\bf{
When $\mu$ is the volume form on $M$ the associated Markov Process $\{X_t;\,t\geq 0\}$ is invariant for the flow $\Theta_s$, $s \geq 0.$ The probability $\bb P_\mu$, defined in the sense of \eqref{explain_measure}, on the Skorohod space obtained from the process will play the role of the {\it a priori} probability   (in a similar way as in \cite{KLMN} and \cite{BEL}). 
}}

General results for continuous-time Markov chains that were specially designed to be used in our setting appear on {\bf Section 1 in \cite{AG}}.

 Note that when considering continuous-time Markov chains the paths are not continuous (in the case of the Brownian motion and diffusions the paths are continuous). The work \cite{KLMN} presents results somewhat similar to those we will describe here, but there the authors consider the continuous-time Markov chains

Let $V:M\to\bb R$ a Lipschitz function and consider the operator $L+V$,
which acts on functions $f\in C^2(M)$ by the expression
\begin{equation*}
(L+V)(f) (x)\,= \frac{1}{2}\pfrac{\partial^2}{\partial x^2}f(x)+\,V(x)f(x) \,,
\end{equation*}
for all $x\in M$.

It is a classical result that there exists a positive differentiable eigenfunction $F: M \to \mathbb{R}$ associated with an eigenvalue $\lambda_V$ (the smallest) for the above operator $L+V$ (see Proposition 2.9 Chapter 8 in \cite{Tay2}, \cite{DV} or \cite{LeSa}).

For $t\geq 0$,  consider
\begin{equation}\label{0}
P_{t}^V (f)(x)\,:=\, \bb E_{x} \big[e^{\int_0^{t} V(X_r)\,dr} f(X_t)\big]\,,
\end{equation}
for all continuous function $f:M\to \bb{R}$ and $x\in M$. By Feynman-Kac, $\{P_{t}^V,\,t\geq 0\}$
defines a semigroup associated with the infinitesimal operator $L+V$ (see Chapter 11 in \cite{Tay2}) which is not Markovian (stochastic). {\bf{The  Feynman-Kac formula is the main inspiration for our reasoning here.}}

Given such $V$ we can normalize this semigroup (associated with the infinitesimal operator $L+V$) to get a new Markov semigroup. This is done via a kind of {\it coboundary procedure} using the positive eigenfunction $F$ and the eigenvalue (for the analogous discrete-time procedure see \cite{PP}). We call the new associated stationary Markov Process - get in this way - the Gibbs Process associated with the perturbation $V$. The shift-invariant probability on the Skorohod space $\mathcal{C}$ obtained from the Gibbs process for $V$ will be called the Gibbs state probability associated with the potential $V$ (see also \cite{BEL} and \cite{KLMN}).

{\bf By \cite{Stri} there is a volume form $\mu$ on $M$ such that the infinitesimal generator $L$  is selfadjoint in $L^2(\mu)$, then we obtain the Brownian motion stochastic process $X_t$, $t\geq 0$, on $M$, with stationary initial distribution $\mu$, which induces a invariant probability $\bb P_\mu$, defined in the sense of \eqref{explain_measure}, on the Skorohod space.
In Appendix \ref{a1} we show that for any continuous functions $f:M \to \mathbb{R}$,  $g:M \to \mathbb{R}$ and $t>0$}

\begin{equation}\label{sim}
\begin{split}
    \int_{\mc C}e^{\int_0^{t} V(w(r))\,dr} &f(w(t))\,\textbf{1}_{w(0)=x}\,d\bb P_{\mu}(w) \\&=\int_{\mc C} e^{\int_0^{t} V(w(r))\,dr} f(w(0))\,\textbf{1}_{w(t)=x}\,d\bb P_{\mu}(w) \,.
\end{split}
\end{equation}

In Figure 
\ref{finalFK}, we show schematically the difference between
looking at the integration of paths by $\bb P_\mu$ via the left-hand side or the right-hand side of
\eqref{sim}; {\bf note that on the right side of \eqref{sim} we take $w(t)=x$, and not $w(0)=x$}. It is important to notice that the Laplacian operator is selfadjoint (see \cite{Stri}). It means that the continuous-time system is reversible. For an explanation of the interest in reversibility in Statistical Mechanics (for continuous Markov chains taking values on a finite set) see Section 4.4 in \cite{Vila}.

The right-hand side of \eqref{sim} is the natural generalization to a continuous-time setting of the classical discrete-time Ruelle operator {\bf (as presented in \cite{ACR}, \cite{LMMS}, and \cite{PP})}. We elaborate on this claim.

Points in the symbolic space $\{1,2,..,d\}^\mathbb{N}$ are denoted by $$x=(x_0,x_1,x_2,..,x_n,...),$$ $x_j \in \{1,2,...,d\}.$

Given a Holder continuous potential $A:\{1,2,...,d\}^\mathbb{N}\to \mathbb{R}$, the discrete time Ruelle operator $\mathcal{L}_A$ acts on continuous functions $\psi: \{1,2,..,d\}^\mathbb{N}\to \mathbb{R}$ via
$$\varphi(x_0,x_1,x_2,x_3,...)\,=\,\mathcal{L}_A (\psi)(x)= $$
$$ \sum_{a=1}^d e^{A(a,x_0, x_1,x_2,x_3,...)}\psi(a,x_0,x_1,x_2,x_3,...)=\sum_{\sigma(y)=x} e^{A(y)} \psi(y),$$
where $\sigma$ is the discrete time shift acting on $\{1,2,..,d\}^\mathbb{N}.$ In this case the {\it a priori} measure is the counting measure on $\{1,2,...,d\}$ (see \cite{LMMS}). Given $n \in \mathbb{N}$
\begin{equation}\label{simn}\mathcal{L}_A^n (\psi)(x)= \sum_{\sigma^n(y)=x} e^{A(y)+ A(\sigma(y) + ...+ A(\sigma^{n-1} (y))} \psi(y).
\end{equation}

Comparing the two expressions, it is fair to say that the right-hand side of  \eqref{sim} is a continuous time version of \eqref{simn} in the case the function $A(x)=A(x_0,x_1,x_2,..,x_n,...)$ depends just on the first coordinate $x_0.$

In Section \ref{um} we present a careful analysis of the main properties  of the continuous-time Ruelle operator, and in Section \ref{dois} we consider relative entropy and a variational principle of pressure. In this way the key elements of the discrete time Thermodynamic Formalism are also present in our continuous time setting.

For related results and applications to Physics see \cite{Kif}, \cite{Kif1}, \cite{DV}, \cite{LMST}, \cite{Gomes} and  \cite{LT}. {\bf  The notion of relative entropy presented here is very close to the respective
notion in statistical mechanics as described in \cite{Georgii}. A definition of entropy for continuous time Markov chains taking values on a finite set is presented in \cite{Dumi} and \cite{Gira}.}

\section{On the continuous time Gibbs state for the potential $V$} \label{um}

General references for basic results on diffusions and semigroups that we use here appear, for example, in \cite{Tay2}, \cite{AG}, \cite{KT2}, \cite{AGM}, \cite{Bo} and \cite{KS}.

Let $\lambda_V$ be the biggest eigenvalue of $L+V$ and $F_V$ the differentiable eigenfunction associated with $\lambda_V$, then $F_V>0$ (for the existence theorems see \cite{AGM}, \cite{EgoS} or \cite{Tay2}). To simplify the notation we will denote $F_V$ by $F$.
For $t\geq 0$,  one defines
\begin{equation}\label{qv}
\mc P^{V}_t (f) (x)\,=\,\bb E_x\Big[ e^{\int_0^{t}\, V(X_r) dr}\, \pfrac{F(X_t)}{e^{{\lambda_V} t}\, F(x)}\, f( X_t)\Big]=\frac{P^{V}_{t} (Ff) (x)}{e^{{\lambda_V}\, t}\, F(x)}\,,
\end{equation}
where $F$ and $\lambda_V$ are  the eigenfunction and the eigenvalue, respectively; {\bf that is $\mc P^{V}_t (F)= e^{\lambda_V\, t} F$}.
Then $\mc P^{V}_t( 1)(x) =1$, $\forall x \in M$.
This defines a stochastic semigroup, which is what we were looking for. From this, we will get a new continuous-time Markov process, which will help
to define the Gibbs state for $V$.

\begin{proposition} \label{LV} Assume  $V: M\to\bb R$ is a Lipschitz function and we define the operator $\mc L_{V}$ acting on $f\in C^2( M)$
as
$$ \mc L_{V}(f)(x)=$$
\begin{equation}\label{gerVF}
\frac{1}{F(x)}(L+V)(Ff)(x)-f(x)\lambda_V=\frac{1}{2}\pfrac{\partial^2}{\partial x^2}f(x)+\pfrac{\partial}{\partial x}\log(F(x))\pfrac{\partial}{\partial x}f(x)\,.
\end{equation}
Then, this operator, $\mc L_V$,
is the infinitesimal generator associated with a semigroup $\{\mc P^{V}_t,\,t\geq 0\}$ defined in
\eqref{qv}.
\end{proposition}

Notice that a process induced by this kind of infinitesimal generator corresponds to a Brownian Motion with non-homogeneous drift: $\pfrac{\partial}{\partial x}\log(F(x))$.

\begin{proof}
It is easy to see that
$\{\mc P^{V}_T,\,T\geq 0\}$ is a semigroup.
To prove that the infinitesimal generator \eqref{gerVF} is associated with this semigroup, we need to observe that
\begin{equation*}
\begin{split}
\frac{ \mc P^{V}_{t} ( f) (x)-f(x)}{t} = \frac{1}{e^{{\lambda_V} t} F(x)}\Bigg(\frac{P^{V}_{t} ( Ff) (x)-(Ff)(x)}{t} \Bigg)
+f(x)\Bigg(\frac{e^{-{\lambda_V} t}-1}{t}\Bigg)\,.
\end{split}
\end{equation*}
Taking the limit as $t$ goes to zero the expression above converges to
\begin{equation*}
\begin{split}
\frac{1}{F(x)}(L+V)(Ff)(x)-f(x)\lambda_V
\,.
\end{split}
\end{equation*}
In order to find the second expression in the statement of the proposition, we can rewrite the expression above as
$$\pfrac{1}{F(x)}\,\frac{1}{2}\pfrac{\partial^2}{\partial x^2}(Ff)(x)+(V(x)-\lambda_V)\,f=$$
\begin{equation}\label{LV1}
\frac{1}{2}\pfrac{\partial^2}{\partial x^2}f(x)+\pfrac{\pfrac{\partial}{\partial x} F(x)}{F(x)} \pfrac{\partial}{\partial x} f+\Big(V(x)-\lambda_V+\frac{1}{2}\pfrac{\ppfrac{\partial^2}{\partial x^2}F(x)}{F(x)}\Big)f(x)\,.
\end{equation}
Using that $F$ is an eigenfunction associated with the eigenvalue $\lambda_V$, we have that
$$V(x)-\lambda_V= -\frac{1}{2}\pfrac{\partial^2}{\partial x^2}F(x)/F(x)\,.$$
Then the last expression in \eqref{LV1} becomes the operator $\mc L_V(f)(x)$, defined in \eqref{gerVF}.

\end{proof}

From now on, we will elaborate on the properties of initial invariant probability $\mu_{V}$, for the operator $\mc L_V$. In other words, $\mu_V$ {\bf will be a probability in $ M$ such that,
for any $f\in\mc C^2( M)$, we have for any $t\geq 0$}
\begin{equation*}
\int \mc P^{V}_t(f) \, \,\mbox{d}\mu_{V}\,=\,\int f\, \,\mbox{d}\mu_{V}\,
\quad
\mbox{or equivalently} \quad
\int \mc L_{V}(f) \, \,\mbox{d}\mu_{V}\,=\,0\,.
\end{equation*}

Since $\mc L_V(f)(x)=\frac{1}{2}\pfrac{\partial^2}{\partial x^2}f(x)+\pfrac{\partial}{\partial x}\log(F(x))\pfrac{\partial}{\partial x}f(x)$ (see \eqref{gerVF}), the following lemma will give us the invariant measure.

\begin{lemma}\label{inv}
Let $G\in C^1( M)$ and define an operator $A: C^2( M)\to\bb R$ as
$$A(f)=\frac{1}{2}\pfrac{\partial^2}{\partial x^2}f+\pfrac{\partial}{\partial x}G\pfrac{\partial}{\partial x}f,$$
for all $f\in C^2( M)$. Then a measure $\mu$ such that $\frac{d\mu}{dx}=e^{2G}$ satisfies
$$\int Af\,d\mu=0,$$
for all $f\in C^2( M)$.
\end{lemma}
\begin{proof}
This proof follows from the Radon-Nikodym Theorem and integration by parts.
\end{proof}

Thus, taking $G=\log F$, we get that $\tilde{\mu}_V$ {\bf  defined by $\frac{d\tilde{\mu}_V}{dx}=F^2$ is the invariant measure for $\mc L_V$.
This measure is not necessarily a probability,} then we will consider the normalized measure
\begin{equation}\label{mu_V}
d\mu_V(x)=\frac{F^2(x)}{\gamma_V}\,dx\,,
\end{equation}
where $\gamma_V=\int_{ M}F^2(x)\,dx$.

\begin{remark}
There is another way to find an invariant measure for $\mc L_V$. One can find first  {\bf an eigenprobability} $\nu_V$ of $L^*+V$ associated with eigenvalue $\lambda_V$.
Then consider
\begin{equation*}
\mu_{V}=F\nu_V,
\end{equation*} where $F$ is be the eigenfunction associated with eigenvalue $\lambda_V$. Since $\mc L_{V}(f) (x)=\frac{1}{F(x)}(L+V)(Ff)(x)-f(x)\lambda_V$, we have
\begin{equation} \label{ga}
\int \mc L_{V}(f) \, \,\mbox{d}\mu_{V}\,=\,\int\Big((L+V)(Ff) -Ff\,\lambda_V\Big)\, \,\mbox{d}\nu_{V}\,=\,0\,.
\end{equation}

\end{remark}

\begin{definition} Given a Lipschitz function $V: M\to\bb R$,
we define a continuous-time Markov process $\{Y^{V}_t, t\geq 0\}$ with state-space $ M$ whose
infinitesimal generator $\mc L_{V}$ acts on functions $f\in C^2( M)$ by the expression
\eqref{gerVF} and the initial stationary probability $\mu_{V}$ defined in \eqref{mu_V}.
We call this process $\{Y^{V}_t, t\geq 0\}$ the continuous time Gibbs state for the potential $V$.
This process induces a probability $\bb P^V_{\mu_{V}}$ on the space $\mc C$, which we call the Gibbs probability
for the potential $V$.
\end{definition}

\begin{remark} Suppose $V$ is of class $C^\infty$ and has a finite number of points with derivative zero.
Let $\lambda$ be the biggest eigenvalue of $L+V$ and $F$ be the eigenfunction associated with $\lambda$.
One can show an interesting property relating oscillations of $V$ and the oscillations of the main eigenfunction $F$.

Suppose that $V:\bb S^1\to \mathbb{R}$ has only two points with derivative zero ($V$ has a unique point of maximum and a unique point of minimum). Then, the eigenfunction $F$ has less than four points with derivative zero.

From the hypothesis on given a value $c$ there exist at most two values $x$ such that $V(x)=c.$
Suppose $F$ has many values with derivative zero. Then, between each two of these points there exists another one $x_1$ with
$F''(x_1)=0$. From, $F''(x_1) + V(x_1)F(x_1) = \lambda F(x_1)$ we get that $V(x_1)=\lambda$. But, by hypothesis one can get at most two points with this property.

One can generalize this for $V$ with more oscillations in a similar way. The analogous property for potentials and eigenfunctions in the setting where the state space has no differentiable structure is not so clear how to get it.

\end{remark}

\medskip

\section{Relative Entropy, Pressure and the equilibrium state for the potential $V$} \label{dois}

One can ask: \textquotedblleft Does the Gibbs state (of the last section) satisfy a variational principle?" We will address this question in the present section.

Given a Lipschitz function $V: M\to \bb R$, we
will consider a {\it variational problem in the continuous-time setting},
which is analogous to the pressure problem in the discrete-time setting (thermodynamic formalism).
This requires a meaning for {\it entropy}. A continuous-time stationary Markov process,
which maximizes our variational problem, will be the {\it continuous-time equilibrium state for $V$}.
The different probabilities $\tilde{\bb P}_{x} $
on $\mc C$ will describe the possible candidates for being the {\it stationary equilibrium continuous-time Markov process for $V$}. These probabilities will be called admissible.

First of all, we will analyze the Radon-Nikodym derivative of $\bb P^V_{x}$ {\bf (obtained from the generator $\mathcal{L}^V$ and the initial probability $\delta_x$, where $x \in M$)} with respect to the measure $\bb P_{x}$ induced by
the Brownian Motion and initial probability $\delta_x$ restricted to $\mc F_t$, where $\{\mc F_t,\,t\geq 0\}$ is the canonical filtration for the Brownian Motion $\{X_t,\,t\geq 0\}$. The Radon-Nikodym derivative will be denoted as $\frac{d\bb P^V_{x}}{d\bb P_{x}}\Big|_{\mc F_t}$, {\bf where $\delta_x$, for any $x\in M$, is taken as  the initial probability in $ M$}.
In order to find an expression to $\frac{d\bb P^V_{x}}{d\bb P_{x}}\Big|_{\mc F_t}$, we remember that it must satisfy
\begin{equation*}
\begin{split}
&\int_{\mc C} G(w_{T_1},w_{T_2},\dots, ,w_{T_k})\,d\bb P^V_{x}(w)=
\int_{\mc C} G(w_{T_1},w_{T_2},\dots, ,w_{T_k})\,\pfrac{d\bb P^V_{x}}{d\bb P_{x}}\Big|_{\mc F_t}\,d\bb P_{x}(w),
\end{split}
\end{equation*}
for all $k\in \bb N$, $0=T_0< T_1<\dots<T_k= t<T$ and $G:( M)^k\to \bb R$ {\bf measurable}.
For this is enough to consider, for any $k\in\bb N$, functions $f_i: M\to \bb R$, $i\in\{1,\dots, k\}$, a time partition as above and, {\bf given $x \in M$}, study the following integral: 
{\tiny \begin{equation*}
\begin{split}
&\int_{\mc C} f_1(w_{T_1})f_2(w_{T_2})\dots f_k(w_{T_k})\,d\bb P^V_{x}(w)\\&= \int_{ M}P_{T_1}(x,x_1)f_1(x_1)\int_{ M}P_{T_2-T_1}(x_1,x_2)f_2(x_2)\dots\int_{ M}P_{T_k-T_{k-1}}(x_{k-1},x_k)f_k(x_k) d x_1...d x_k\\
&= \int_{ M}P_{T_1}(x,x_1)f_1(x_1)\dots\int_{ M}P_{T_{k-1}-T_{k-2}}(x_{k-2},x_{k-1})f_{k-1}(x_{k-1})\mc P^V_{T_k-T_{k-1}}(f_k)(x_{k-1})d x_1...d x_k\\
&=\cdots\\
&= \,\mc P^V_{T_1}(f_1\dots(\mc P^V_{T_{k-1}-T_{k-2}}(f_{k-1}\mc P^V_{T_k-T_{k-1}}(f_k)))\dots)(x)
\end{split}
\end{equation*},}
{\bf where $P_t (x,y)$ is the transition kernel for $\mathcal{P}_t^V$ (see \cite{AGM})}

To fix ideas consider $k=2$, {\bf two times} $T_1<T_2$, and analyze:
\begin{equation*}
\begin{split}
&\mc P^V_{T_1}\big(f_1(\mc P^V_{T_{2}-T_{1}}(f_{2})\big)(x),
\end{split}
\end{equation*}where for $T_1<T_2$

{\tiny
\begin{equation*}
\begin{split}
&\mc P^V_{T_1}\big(f_1(\mc P^V_{T_{2}-T_{1}}(f_{2})\big)(x)\\&=
\bb E_{x}\Bigg[ e^{\int_0^{T_1}\, V(X_r) dr}\, \frac{F(X_{T_1})}{e^{{\lambda_V} T_1}\, F(x_0)}\, f_1( X_{T_1})\;\bb E_{X_{T_1}}\Big[ e^{\int_0^{T_2-T_1}\, V(X_r) dr}\, \frac{F(X_{T_2-T_1})}{e^{{\lambda_V} (T_2-T_1)}\, F(X_{T_1})}\, f_2( X_{T_2-T_1})\Big]\Bigg]\\
\\&=
\bb E_{x}\Bigg[ e^{\int_0^{T_1}\, V(X_r) dr}\, \frac{1}{e^{{\lambda_V} T_2}\, F(x_0)}\, f_1( X_{T_1})\;\bb E_{X_{T_1}}\Big[ e^{\int_0^{T_2-T_1}\, V(X_r) dr}\, F(X_{T_2-T_1})\, f_2( X_{T_2-T_1})\Big]\Bigg]\\
&=^{using Markov Property}
\bb E_{x}\Bigg[ e^{\int_0^{T_1}\, V(X_r) dr}\, \frac{1}{e^{{\lambda_V} T_1}\, F(x_0)}\, f_1( X_{T_1})\; e^{\int_{T_1}^{T_2}\, V(X_r) dr}\, F(X_{T_2})\, \, f_2( X_{T_2})\Bigg]\\
&=
\bb E_{x}\Bigg[f_1( X_{T_1})\, f_2( X_{T_2})\; e^{\int_0^{T_2}\, V(X_r) dr}\, \frac{F(X_{T_2})}{e^{{\lambda_V} T_2}\, F(x_0)}\, \Bigg].
\end{split}
\end{equation*}}

In this case,
\begin{equation*}
\begin{split}
&\int_{\mc C} f_1(w_{T_1})f_2(w_{T_2})\,d\bb P^V_{\mu}(w)\\&=\int_{\mc C} f_1( w_{T_1})\, f_2( w_{T_2})\, e^{\int_0^{T_2}\, V(w_r) dr}\, \pfrac{F(w_{T_2})}{e^{{\lambda_V} T_2}\, F(w_0)}\,d\bb P_{\mu}(w)\,.
\end{split}
\end{equation*}
Remembering that $T_k=t$, we have that, for all $T>t\geq 0$,
$$ \pfrac{d\bb P^V_{x}}{d\bb P_{x}}\Big|_{\mc F_t}(w)\,=$$
\begin{equation}\label{rn}
\,\exp\Big\{\,\log F(w_{t})\,-\,\log F(w_0)\,-\int_0^{t}\, (\lambda_V-V(w_r)) \,dr\,\Big\},\quad\bb P_{x}-a.s.
\end{equation}
or, using another notation,
\begin{equation*}
\pfrac{d\bb P^V_{x}}{d\bb P_{x}}\Big|_{\mc F_t}\,=\,\exp\Big\{\,\log F(X_{t})\,-\,\log F(X_0)\,-\int_0^{t}\, (\lambda_V-V(X_r)) \,dr\,\Big\},
\end{equation*}
where $\{X_s,s\geq 0\}$ is the Brownian Motion.

\begin{definition}\label{admissible}
The probability $\tilde{\bb P}_{\mu} $ on $\mc C$
is called admissible, if, for all $T\geq 0$,
\begin{equation}\label{rntilde}
\pfrac{d\tilde{\bb P}_{x} }{d\bb P_{x} }\Big|_{\mc F_T}\,=\,\,\exp\Big\{g(X_T)-g(X_0)- \frac{1}{2}\int_{0}^T\Big[ \pfrac{\partial^2}{\partial x^2}g(X_r)+(\pfrac{\partial}{\partial x}g)^2(X_r)\Big]\,dr\,\Big\}\,,
\end{equation}
for some function $g\in C^2( M)$.
\end{definition}


Notice that according to the above the Gibbs Markov process $\bb P^V_{x}$ with the initial probability $\delta_x$
is admissible, it is enough to take $g=\log F$ and observe that
\begin{equation}\label{pop}
\frac{1}{2}[\pfrac{\partial^2}{\partial x^2}g+(\pfrac{\partial}{\partial x}g)^2]
= \frac{1}{2}[\pfrac{\partial}{\partial x}\Big(\frac{\pfrac{\partial}{\partial x}F}{F}\Big)+\Big(\frac{\pfrac{\partial}{\partial x}F}{F}\Big)^2]=\frac{\frac{1}{2}\,\pfrac{\partial^2}{\partial x^2}F}{F}=\frac{LF}{F}=\lambda_V-V\,.
\end{equation}
The last equality is due to $(L+V)F=\lambda_VF$.
And, the unperturbed system with the initial measure $\delta_x$ is also admissible, just take $g=0$.

Denote by $\{\tilde{X}_t,\,t\geq 0\}$ which has law in $\mc C$ the probability $\tilde{\bb P}_x$.
The question is: who is the process $\{\tilde{X}_t,\,t\geq 0\}$?
The answer is in \cite[Chapter VIII.3]{RY} thanks to Girsanov's Theorem.
More specifically, by Proposition 3.4 in \cite[Chapter VIII.3]{RY} the infinitesimal generator of the process $\{\tilde{X}_t,\,t\geq 0\}$ is
$\tilde{L}=L+\Gamma(g,\cdot)$, where $\Gamma$ is the {\emph{op\'erateur carr\'e du champ}} defined in $C^2( M)\times C^2( M)$ as
$$\Gamma(f,g)= L(fg)-fLg-gLf.$$
Since
$L=\frac{1}{2}\,\frac{\partial^2}{\partial x^2}$, the {\emph{op\'erateur carr\'e du champ}} is just
$$\Gamma(f,g)= \pfrac{\partial}{\partial x}g\pfrac{\partial}{\partial x}f.$$
Thus the Radom-Nikodym derivative $\frac{d\tilde{\bb P}_{x} }{d\bb P_{x} }$, defined in \eqref{rntilde}, can be rewritten as
\begin{equation*}
\pfrac{d\tilde{\bb P}_{x} }{d\bb P_{x} }\Big|_{\mc F_t}\,=\,\,\exp\Big\{g(X_t)-g(X_0)-\int_{0}^t\big[ Lg(X_r)+\Gamma(g,g)(X_r)\big]\,dr\,\Big\}\,.
\end{equation*}
By \cite{RY}, if the process has a Radon-Nikodym derivative with respect to $\bb P_x$ as above, the generator of this process is $\tilde{L}=L+\Gamma(g,\cdot)$.

The conclusion is that in our model the generator of $\{\tilde{X}_t,\,t\geq 0\}$ acts on functions $f\in C^2( M)$ as
\begin{equation}\label{tilde L}
\tilde{L}f= \frac{1}{2}\pfrac{\partial^2}{\partial x^2}f+\pfrac{\partial}{\partial x}g\pfrac{\partial}{\partial x}f.
\end{equation}
Then, the process $\{\tilde{X}_t,\,t\geq 0\}$ is a Brownian Motion with drift $\pfrac{\partial}{\partial x}g$, i.e., this process is in the same class of the Gibbs Markov process $\bb P^V_{x}$ as it should be.

{\bf Now we will find an invariant measure for $\tilde{L}$, which we will denote by $\tilde{\mu}$.
By Lemma \ref{inv}, the invariant measure $\tilde{\mu}$ for $\tilde{L}$ is such that $d\tilde{\mu}(x)=e^{2g(x)}/\tilde{\gamma}\,dx$, where $\tilde{\gamma}=\int_{ M}e^{2g(x)}\,dx$.}

\bigskip

Now, we want to give a meaning for the relative entropy of any admissible probability $\tilde{\bb P}_{\tilde{\mu}}$ with respect to $\bb P_{\tilde{\mu}}$, \textbf{where} 
$\tilde{\bb P}_{\tilde{\mu}}[A]=\int_M \tilde{\bb P}_{x}[A]\,d\tilde{\mu}(x)$ and ${\bb P}_{\tilde{\mu}}[A]=\int_M {\bb P}_{x}[A]\,d\tilde{\mu}(x)$, \textbf{for all mensurable set $A$ in the Skorohod space. Observe that, if we choose the initial measure $\tilde{\mu}$ on $M$ as $\delta_{x_0}$ (the Dirac measure at the point $x_0\in M$, then  we have ${\bb P}_{\delta_{x_0}}={\bb P}_{x_0}$ what totally agrees with we introduced in the beginning of the Section 2. }

The reason why we use the same initial measure for both processes is that
we need that the associated probabilities, $\tilde{\bb P}_{\tilde{\mu}}$ and $\bb P_{\tilde{\mu}}$,
on $\mc C$ are absolutely continuous with respect to each other.
Anyway, the final numerical result for the value of entropy will not depend on the common $\tilde{\mu}$ we chose as the initial probability.

For a fixed $T\geq 0$, we consider the relative entropy
of the $\tilde{\bb P}_{\tilde{\mu}}$,
with respect to $\bb P_{\tilde{\mu}}$ up to time $T\geq 0$
as
\begin{equation}\label{entropy}
H_T(\tilde{\bb P}_{\tilde{\mu}}\vert\bb P_{\tilde{\mu}})\,=-\,\int_{ M}\int_{\mc C}
\log\Bigg(\frac{\mbox{d}\tilde{\bb P}_x}{\mbox{d}\bb P_x}\Big|_{\mc F_T}\Bigg)(\omega)\,
\mbox{d}\tilde{\bb P}_x(\omega)\,\mbox{d}{\tilde{\mu}}(x)\,.
\end{equation}

Using the property that the logarithm is a concave function and Jensen's inequality, we obtain that
for any $g$ we have $\int \log g \,\mbox{d} \mu\leq \log \int g \,\mbox{d} \mu$. Then
$H_T(\tilde{\bb P}_{\tilde{\mu}}\vert\bb P_{\tilde{\mu}})\leq 0$.
Negative entropies appear naturally when one analyzes a dynamical system with the property that
each point has an uncountable number of preimages (see \cite{LMMS} and \cite{LMS}).

Using the expression \eqref{rntilde}, we can rewrite the entropy $H_T(\tilde{\bb P}_{\tilde{\mu}}\vert\bb P_{\tilde{\mu}})$ as
\begin{equation}\label{17}
\begin{split}
\,\,&-
\int_{\mc C}\Big[g(w_T)-g(w_0)-\frac{1}{2}\,\int_{0}^T\Big[ \pfrac{\partial^2}{\partial x^2}g(w_r)+(\pfrac{\partial}{\partial x}g)^2(w_r)\Big]\,dr\Big]\,d\tilde{\bb P}_{\tilde{\mu}}(w)\\=\,&\int_{ M}\Big\{\tilde{P}_0g(x)-\tilde{P}_Tg(x)+\frac{1}{2}\int_0^T \tilde{P}_r[\pfrac{\partial^2}{\partial x^2}g+(\pfrac{\partial}{\partial x}g)^2](x)\,dr\Big\}\,d {\tilde{\mu}}(x)\,,
\end{split}
\end{equation}
where $\tilde{P}_t$ is the semigroup associated with $\tilde{L}$.

From the previous expression and ergodicity, we get that there exists the limit $\lim_{T\to \infty} \frac{1}{T} H_T(\tilde{\bb P}_{\tilde{\mu}}\vert\bb P_{\tilde{\mu}})$.

\begin{definition}\label{12}
For a fixed initial probability ${\tilde{\mu}}$ on $ M$, we will denote the limit
\begin{equation*}
\lim_{T\to \infty} \pfrac{1}{T} H_T(\tilde{\bb P}_{\tilde{\mu}}\vert\bb P_{\tilde{\mu}})
\end{equation*}
as $H(\tilde{\bb P}_{\tilde{\mu}}\vert\bb P_{\tilde{\mu}})$.
Moreover, we will call $H(\tilde{\bb P}_{\tilde{\mu}}\vert\bb P_{\tilde{\mu}})$ as the \emph{relative entropy} of the measure $\tilde{\bb P}_{\tilde{\mu}}$ with respect to the measure $\bb P_{\tilde{\mu}}$.
\end{definition}

By the Definition \ref{12}, the expression \eqref{17} and the Ergodic Theorem {\bf (see Theorem 1.14 in \cite{Loche})}, the relative entropy
\begin{equation*}
\begin{split}
H(\tilde{\bb P}_{\tilde{\mu}}\vert\bb P_{\tilde{\mu}})=\frac{1}{2}\int_{ M}\Big[\pfrac{\partial^2}{\partial x^2}g+(\pfrac{\partial}{\partial x}g)^2\Big]\,d\tilde{\mu}\,.
\end{split}
\end{equation*}

\begin{definition} For a given Lipschitz potential $V$, we denote the Pressure (or, Free Energy) of $V$ as the value
\begin{equation*}
\textbf{P}(V):= \sup_{\at{\tilde{\bb P}_{\tilde{\mu}}}{\text{admissible}}}\Big\{
\,H(\tilde{\bb P}_{\tilde{\mu}}\vert\bb P_{\tilde{\mu}})\, +\,
\int_{ M}V\, \mbox{d}\tilde{\mu}\Big\}\,,
\end{equation*}
where $\tilde{\mu}$ is the initial stationary probability for the infinitesimal generator $\tilde{L}$, defined in \eqref{tilde L}.
Moreover, any admissible element that maximizes $\textbf{P}(V)$ is called a continuous-time equilibrium state for $V$.
\end{definition}

Finally, we can state the main result of this section:
\begin{proposition}
The pressure of the potential $V$ is given by
\begin{equation*}
\textbf{P}(V)\,=\,H(\bb P^V_{\mu_V}\vert\bb P_{\mu_V})+
\int_{ M}V\,\mbox{d}\mu_{V}=\lambda_{V}\, .
\end{equation*}

Therefore, the equilibrium state for $V$ is the Gibbs state for $V$.
\end{proposition}

\begin{proof}

The second equality in the statement of the theorem comes from
$$H(\bb P^V_{\mu_V}\vert\bb P_{\mu_V})+
\int_{ M}V\,\mbox{d}\mu_{V}=  $$
\begin{equation*}
\int_{ M}\frac{1}{2}\Big[\pfrac{\p^2}{\p_x^2}\log F+(\pfrac{\p}{\p_x}\log F)^2]+V\Big]\,\mbox{d}\mu_{V}=\int_{ M}\pfrac{LF}{F}+V\,\mbox{d}\mu_{V}=\lambda_V\,,
\end{equation*}
by \eqref{pop}  and $(L+V)F=\lambda_VF$.

To finish the proof, we need to analyze
\begin{equation}\label{enttilde}
\,H(\tilde{\bb P}_{\tilde{\mu}}\vert\bb P_{\tilde{\mu}})\, +\,
\int_{ M}V\, d\tilde{\mu},
\end{equation}
which is equal to
\begin{equation*}
\begin{split}
    &\frac{1}{\tilde{\gamma}}\frac{1}{2}\int_{ M}\Big[\pfrac{\partial^2}{\partial x^2}g+(\pfrac{\partial}{\partial x}g)^2\Big]\,e^{2g}\,dx\, +\,
\frac{1}{\tilde{\gamma}}\int_{ M}V\,\,e^{2g}\,dx\\=&\frac{1}{\tilde{\gamma}}\int_{ M}\Big[V-\frac{1}{2}(\pfrac{\partial}{\partial x}g)^2\Big]\,e^{2g}\,dx.
\end{split}
\end{equation*}
The last equality follows from integration by parts and the expression of $\tilde{\mu}$.
Using that $V=\lambda_V-\pfrac{LF}{F}$, we can rewrite the last integral above as
\begin{equation*}
\lambda_V+\frac{1}{\tilde{\gamma}}\frac{1}{2}\int_{ M}\Big[-\frac{\pfrac{\partial^2}{\partial x^2} F}{F}-(\pfrac{\partial}{\partial x}g)^2\Big]\,e^{2g}\,dx.
\end{equation*}
Applying integration by parts, the integral above becomes
\begin{equation*}
\frac{1}{2}[\,\int_{ M}\pfrac{\partial}{\partial x} F\pfrac{\partial}{\partial x} \Big(\frac{e^{2g}}{F}\Big)\,dx\,-
\int_{ M}(\pfrac{\partial}{\partial x}g)^2\,e^{2g}\,dx\,].
\end{equation*}
The expression above can be rewritten as
\begin{equation*}
-\frac{1}{2}\,\int_{ M}\Big(\pfrac{\partial}{\partial x}(\log F)-\pfrac{\partial}{\partial x}g\Big)^2\,e^{2g}\,dx.
\end{equation*}
Therefore, the expression in \eqref{enttilde} is less than or equal to $\lambda_V$.

\end{proof}
\bigskip

\section{Appendix}\label{a1}
\begin{lemma}Let $\mu$ a measure on $M$ such that the infinitesimal generator $L$ is selfadjoint in $L^2(\mu)$, that is,    \begin{equation}\label{sim2}\int_M(Lf)(x)g(x)\,d\mu(x)=\int_Mf(x)(Lg)(x)\,d\mu(x), \end{equation}
for all continuous functions $f,g:M\to \bb R$. Consider $\bb P_{\mu}$ the {\it a priori}
probability, which is induced by the initial measure $\mu$ and the infinitesimal generator $L$. And. denote by $\bb E_{\mu}$ the expectation concerning to the $\bb P_{\mu}$. Then, we have
  \begin{equation}\label{sim3}
  \begin{split}
      \int_{\mc C}e^{\int_0^{t} V(w(r))\,dr}& f(w(t))\,g(w(0))\,d\bb P_{\mu}(w) \\&=\int_{\mc C} e^{\int_0^{t} V(w(r))\,dr} f(w(0))\,g(w(t))\,d\bb P_{\mu}(w) \,,
  \end{split}
\end{equation}
for all continuous functions $f,g:M\to \bb R$.
\end{lemma}

\begin{proof}
We start this proof by observing that the equality \eqref{sim3} can be rewritten as 
    \begin{equation}\label{sim1}\bb E_{\mu} \Big[e^{\int_0^{t} V(X_r)\,dr} f(X_t)\,g(X_0)\Big] = \bb E_{\mu} \Big[e^{\int_0^{t} V(X_r)\,dr} f(X_0)\,g(X_t)\Big] \,.
\end{equation} Then our goal is to prove \eqref{sim1}, in order to do this we 
    use \eqref{explain_measure} in the left-hand side of \eqref{sim1}, and we obtain
    \begin{equation}\label{Aa1}
    \begin{split}
        \bb E_{\mu} \Big[e^{\int_0^{t} V(X_r)\,dr} f(X_t)\,g(X_0)\Big] &= 
  \int_M  \bb E_{x} \Big[e^{\int_0^{t} V(X_r)\,dr} f(X_t)\,g(X_0)\Big]\,d\mu(x)\\
  &= 
  \int_M  \bb E_{x} \Big[e^{\int_0^{t} V(X_r)\,dr} f(X_t)\Big]\,g(x)\,d\mu(x)\\
  &=\int_M (P^V_tf)(x)\, g(x)\,d\mu(x)\,,
    \end{split}
    \end{equation}
    where the last equality is due to the expression \eqref{0} for the semigroup $P^V_t$, which is associated to the infinitesimal generator $L+V$. By \eqref{sim2}, we have that $L+V$ is selfadjoint, that is, 
    \begin{equation*}
   \begin{split}
       & \int_M(L+V)(f)(x)\;g(x)\,d\mu(x)= \int_M(Lf)(x)g(x)\,d\mu(x)+\int_M V(x)f(x)g(x)\,d\mu(x)\\&=\int_Mf(x)\;(L+V)(g)(x)\,d\mu(x).
   \end{split}
    \end{equation*}
    Then the semigroup $P^V_t$, which is associated with $L+V$, is selfadjoint too. Thus,
      \begin{equation}\label{Aa2}
    \begin{split}
\int_M (P^V_tf)(x)\, g(x)\,d\mu(x)=\int_M f(x)\,(P^V_tg)(x)\,d\mu(x)\,,
    \end{split}
    \end{equation}
    for all continuous functions $f,g:M\to \bb R$. Writing  the semigroup $P^V_t$ with the expression \eqref{0} and  using \eqref{explain_measure}, we get
      \begin{equation}\label{Aa3}
    \begin{split}
\int_M f(x)\,(P^V_tg)(x)\,d\mu(x)&=\int_M f(x)\,\bb E_{x} \Big[e^{\int_0^{t} V(X_r)\,dr} g(X_t)\Big]\,d\mu(x)\\
&=\int_M \bb E_{x} \Big[e^{\int_0^{t} V(X_r)\,dr} g(X_t)\, f(X_0)\Big]\,d\mu(x)\\
&=\bb E_{\mu} \Big[e^{\int_0^{t} V(X_r)\,dr} g(X_t)\, f(X_0)\Big]\,.
    \end{split}
    \end{equation}
    Putting \eqref{Aa1}, \eqref{Aa2}, and \eqref{Aa3} together we obtain \eqref{sim1}.
\end{proof}

\begin{corollary}For all continuous function $f:M\to \mathbb R$, we have
    \begin{equation*}
   \begin{split}
        \int_{\mc C}e^{\int_0^{t} V(w(r))\,dr} &f(w(t))\,\textbf{1}_{w(0)=x}\,d\bb P_{\mu}(w) \\=&\int_{\mc C} e^{\int_0^{t} V(w(r))\,dr} f(w(0))\,\textbf{1}_{w(t)=x}\,d\bb P_{\mu}(w) \,.
   \end{split}
\end{equation*}
\end{corollary}
\begin{proof}
    Using the expression \eqref{sim3} with a smooth function $g_n$ and 
taking $g_n(y)\to \textbf{1}_{x}(y)$, we get the intended result.
\end{proof}


\bigskip



\end{document}